\documentclass[12pt,leqno]{article}

\usepackage{amsmath,amssymb,amsthm}
 
\usepackage[american]{babel}

\usepackage{eucal,euler,times}

\usepackage{enumerate}

\usepackage{xy}
\xyoption{all}

\usepackage{graphicx}
\usepackage{color}

\swapnumbers

\theoremstyle{plain}
    \newtheorem{theorem}{Theorem}[section]
    \newtheorem{lemma}[theorem]{Lemma}
    \newtheorem{corollary}[theorem]{Corollary}
    \newtheorem{proposition}[theorem]{Proposition}

\theoremstyle{definition}
    \newtheorem{definition}[theorem]{Definition}
    \newtheorem{example}[theorem]{Example}

     \newtheorem{remark}[theorem]{Remark}

\theoremstyle{remark}

\numberwithin{equation}{section}

\def\Id{\mathrm {Id}}
\def\R{{\mathbb R}}
\def\Z{{\mathbb Z}}

\def\C{{\mathbb C}}

\DeclareMathOperator{\Spin}{Spin}
 
\DeclareMathOperator{\Ind}{Ind}

\DeclareMathOperator{\ev}{ev}

\DeclareMathOperator{\Pt}{pt}
\newcommand{\geomG}{G}
\DeclareMathOperator{\odd}{odd}
\DeclareMathOperator{\Cliff}{{\C}liff}

\DeclareMathOperator{\rank}{rank}

\newcommand{\Times}{{\times}}

\begin{document}

\title{A Geometric Description of Equivariant K-Homology for Proper Actions}
\author{Paul Baum, Nigel Higson and Thomas Schick\thanks{NH and PB were
    partially supported by grants from the US National Science Foundation.  TS
    was partially supported by the Courant
    Research Center "Higher order structures in mathematics" via the German
    Initiative of Excellence}}
\date{}

\maketitle

\begin{abstract}
Let $G$ be a discrete group and let $X$ be a  $G$-finite,  proper $G$-$CW$-complex.  We prove that Kasparov's equivariant $K$-homology groups $KK^G_*(C_0(X),\C)$ are isomorphic to the  geometric equivariant $K$-homology groups of $X$ that are obtained by making the geometric $K$-homology theory of Baum and Douglas equivariant in the natural way.  This reconciles the original and current formulations of the Baum-Connes conjecture for discrete groups.

\begin{center}
\emph{Dedicated with admiration and affection to \\ Alain Connes on his 60th birthday.}
\end{center}
\end{abstract}

\section{Introduction}

In the original formulation of the Baum-Connes conjecture \cite{MR1769535},
the topological $K$-theory of a discrete group $G$ (the ``left-hand side'' of
the conjecture) was defined geometrically in terms of   proper $G$-manifolds.
Later, in \cite{MR1292018}, the definition   was changed so as to involve
Kasparov's equivariant $KK$-theory.  The change was made to accommodate new
examples beyond the realm of discrete groups, such as $p$-adic groups, for
which the geometric definition was not convenient or adequate. But it left
open the question of whether the original and revised definitions are
equivalent for discrete groups.  This is the question that we  shall
address in this paper.

In a recent article \cite{MR2330153}, we gave a complete proof    that the (non-equi\-variant) geometric $K$-homology theory of Baum and Douglas \cite{MR679698} agrees with Kasparov's $K$-homology on finite $CW$-complexes.  Here we shall show that our techniques   extend   to show that the original and revised definitions of   topological $K$-theory for a discrete group agree---provided that those techniques are supplemented by  a key result  of L\"uck and Oliver about equivariant vector bundles over  $G$-finite, proper  $G$-$CW$-complexes \cite{MR1838997}.

L\"uck and Oliver prove  that if $X$ is a $G$-finite, proper $G$-$CW$-complex, then there is a rich supply of equivariant vector bundles on $X$.  It follows that the Grothendieck group $K^0_G(X)$ of complex $G$-vector bundles  is the degree zero group of a $\Z / 2\Z$-graded cohomology theory on $X$, and it is essentially this   fact that we shall   need   to carry over the arguments of \cite{MR2330153} to the equivariant case.   

We shall show that the L\"uck-Oliver theorem  is equivalent to the assertion that the crossed product $C^*$-algebra  $C^*(X,G)$ associated to the action of $G$ on $X$ has an approximate identity consisting of projections.  As a result
 $K_G^0(X)$ is isomorphic to the $K_0$-group of the crossed product $C^*$-algebra. 
This has some further benefits for us---for example it makes it clear that each  complex $G$-vector bundles on a $G$-compact proper $G$-manifold has a unique smooth structure, up to isomorphism.

Returning to the Baum-Connes conjecture, the assertion that the old and the revised   versions are the same is a consequence of the assertion that the natural $G$-equivariant development of the Baum-Douglas $K$-homology theory, which we shall write as $K^G_*(X)$, is isomorphic to  Kasparov's   group $KK_*^G(C_0(X),\C)$  for any $G$-finite, proper $G$-$CW$-complex $X$.  There is a natural map 
\begin{equation*}
\label{nat-trans-eq}
\mu\colon 
K^{G}_* (X) \longrightarrow KK^{G}_*(C_0(X),\C) 
\end{equation*}
which is defined using the index of Dirac operators, and we shall prove that it is an isomorphism.   What makes this nontrivial is that the groups $K_*^G(X)$ do not obviously constitute a homology theory.  We shall  address this problem by introducing   groups $k_*^G (X)$ that manifestly \emph{do} constitute a homology theory  and by constructing a commuting  diagram 
$$
\xymatrix@C=50pt{
  k^{G}_* (X)\ar[r]\ar@{=}[d] &   K^{G}_* (X)\ar[d]^-{\mu}\\
 k^{G}_*(X) \ar[r]_-\mu&KK^{G}_*(C_0(X ),\C).}$$
We shall prove that the map from $k^G_*(X)$ to $KK^{G}_*(C_0(X ),\C)$  is an
isomorphism when $X$ is a $G$-finite, proper $G$-$CW$-complex and that map
from $k^G_*(X)$ to   $K^G_*(X)$ is   surjective.  This proves that the map
from $ k^{G}_*(X)$  to $KK^{G}_*(C_0(X ),\C)$ is an isomorphism.

A more general, but less concrete model for equivariant K-theory, based on correspondences, is given by
Emerson and Meyer in \cite{Emerson-Meyer}. They  describe in the end a bivariant
K-theory for spaces with proper groupoid actions.  Their theory does not always coincide with
KK-theory, but when their theory overlaps with ours the two coincide.

 In \cite{Baum-Oyono-Schick} it is shown that for compact
  Lie groups the result corresponding to the main result in this paper
  holds (a direct construction is given of an inverse to the
  transformation $\mu$). The basis of the theorem is again the  fact that for compact Lie group actions there are ``enough''
  equivariant vector bundles.

\section{Proper Actions}
\label{proper-sec}

Throughout the paper we shall work with a fixed a countable discrete group $G$.  By a \emph{$G$-space} we shall mean a topological space with an action of $G$ by homeomorphisms. We shall be concerned in the first place with  \emph{proper $G$-$CW$-complexes}. These are $G$-spaces  with  filtrations 
$$\emptyset = X^{-1}\subseteq X^0\subseteq X^1 \subseteq  \cdots \subseteq X$$
such that $X^k$ is obtained from $X^{k-1}$ by attaching equivariant cells of the form $D^k\times G/H$ along their boundaries, where $H$ is any finite subgroup of $G$. See \cite[Section~1]{MR2195456} for more details.

The Baum-Connes conjecture as formulated in \cite{MR1292018} involves \emph{universal proper $G$-spaces}.  In the context of proper   $G$-$CW$-complexes, these may be characterized as follows: 

\begin{theorem}[{\cite[Theorem~1.9]{MR2195456}}] 
\label{ebar-existence-thm}
There is a proper $G$-$CW$-complex $\underline E G$ with the property that if $Y$ is any proper $G$-$CW$-complex, then there is a  $G$-equivariant continuous  map  from $Y$ into $\underline E G$, and moreover this  map  is unique  up to equivariant homotopy.
\end{theorem}

Clearly the $G$-$CW$-complex $\underline E G$ is unique up to equivariant homotopy.   The universal space used in the formulation of the Baum-Connes conjecture is defined a bit differently (see \cite[Section~1]{MR1292018}),  but by the results of \cite[Section~2]{MR2195456} the same conjecture results if the above version of $\underline E G$ is used.  Compare also Theorem~\ref{ebar-numerable-thm} below.

 A proper $G$-$CW$-complex is said to be \emph{$G$-finite}
if only finitely many equivariant cells are used in its construction.  These are $G$-compact proper $G$-spaces  in the sense of the following definition.

\begin{definition}
 We shall say that a $G$-space $X$ is a \emph{$G$-compact, proper $G$-space} if 
\begin{enumerate}[\rm (a)]
\item $X$ is locally compact and Hausdorff.

\item The quotient space $X/G$ is compact and Hausdorff in the quotient topology.

\item Each point of $X$ is contained in an equivariant neighborhood $U$ that maps continuously and equivariantly onto some proper orbit space $G/H$ (where $H$ is a finite subgroup of $G$).
\end{enumerate}
\end{definition}

Apart from $G$-finite $G$-$CW$-complexes, we shall also be concerned with   smooth manifolds (with smooth actions of $G$) that satisfy these conditions. We shall call them \emph{$G$-compact proper $G$-manifolds}.  
 
The following result is a consequence of \cite[Theorem~3.7]{MR2195456} (the final statement reflects a simple feature of the $CW$-topology on $\underline E G$).

\begin{theorem}
\label{ebar-numerable-thm}
If $X$ is any $G$-finite proper $G$-space, then there is a $G$-equivariant map from $X$ to $\underline E G$.  It is unique up to equivariant homotopy, and its image is contained within a $G$-finite subcomplex of $\underline E G$.
 \end{theorem}

\section{Equivariant Geometric K-Homology}
\label{geometric-sec}

In this section we shall present the equivariant version of the geometric $K$-hom\-ology theory of Baum and Douglas \cite{MR679698}.   The definition presents no difficulties,   so we shall be brief.  The reader is referred to \cite{MR679698} or \cite{MR2330153} for treatments of the non-equivariant theory.  

We shall work with principal bundles, rather than with spinor bundles as in \cite{MR679698} or \cite{MR2330153}.
To fix notation, recall  the following rudimentary facts about Clifford algebras, $\Spin^c$-groups and $\Spin^c$-structures.  Denote by $\Cliff (n)$ the $\Z/2\Z$-graded complex $*$-algebra generated by skew-adjoint degree-one  elements $e_1,\dots, e_n$ such that  $$e_ie _j + e_j e _i= -2\delta_{ij} I.$$
  We shall consider $\R^n$ as embedded into $\Cliff (n)$ in such a way that the standard basis of $\R^n$ is carried to $e_1,\dots , e_n$.

Denote by  $\Spin^c(n)$   the group of all even-grading-degree unitary
elements in $ \Cliff (n)$ that map $\R^n$ into itself under the adjoint
action.    This is a compact Lie group.  The image of the group homomorphism 
$$\alpha\colon \Spin^c(n)\longrightarrow GL(n,\R)$$ 
given by the adjoint action is  $SO(n)$  and the  kernel is the circle group $U(1)$ of all unitaries in $\Cliff (n)$ that are multiples of the identity element. 

There is a natural complex conjugation operation on $\Cliff (n)$ (since the relations defining the Clifford algebra involve only real coefficients) and the map $u\mapsto u\bar u^*$ is a homomorphism from  $\Spin^c(n)$ onto $U(1)$.  The combined homomorphism 
$$\Spin^c(n)\longrightarrow SO(n)\times U(1)$$
is a double covering.

Let $M$ be a smooth, proper $G$-manifold and let $V$ be a smooth, real $G$-vector bundle over $M$ of rank $n$.  A \emph{$G$-$\Spin^c$-structure} on $V$ is a homotopy class of  reductions of the principal  frame bundle  of $V$ (viewed as a $G$-equivariant right principal $GL(n,\R)$-bundle)  to a $G$-equivariant principal $\Spin^c(n)$-bundle.  In other words it is a $G$-homotopy class of commuting diagrams 
$$
\xymatrix{
Q\ar[r]^\varphi \ar[d]& P \ar[d]\\
M\ar@{=}[r]& M
}
$$
of smooth $G$-manifolds,
where $P$ is the  bundle of ordered bases for the fibers of $V$, $Q$ is a $G$-equivariant principal $\Spin^c(n)$-bundle, and $$\varphi (q u) = \varphi(q)\alpha(u)$$ for every $q\in Q$ and every $u\in \Spin^c(n)$.  A $G$-$\Spin^c$-structure on $V$ determines a $G$-invariant orientation, and a specific choice of   $Q$ within its   homotopy class determines a Euclidean structure.

\begin{example} Every $G$-equivariant complex vector bundle carries a natural $\Spin^c$-structure because there is a (unique) group homomorphism
$ 
U(k)\to  \Spin^c(2k)$ that lifts the map
$$
U(k)\longrightarrow SO(2k)\times U(1)
$$
given  in the right-hand factor by the determinant.
\end{example}

  A  \emph{$G$-$\Spin^c$-vector bundle} is a smooth real $G$-vector bundle with a given $G$-$\Spin^c$-structure.   The direct sum of two $G$-$\Spin^c$-vector bundles carries a natural $G$-$\Spin^c$-struc\-ture.  It is obtained from the diagram 
$$
\xymatrix{ \Spin^c(m)\times \Spin^c(n)\ar[d]\ar[r] & \Spin^c (m+n) \ar[d]\\
GL(n,\R)\times GL(m,\R)\ar[r] & GL(m+n,\R)
}
$$
that is in turn obtained from the inclusions of $\Cliff (m)$ and $\Cliff (n)$ into $\Cliff (m+n)$   given by the formulas $e_k\mapsto e_k$ and $e_k\mapsto e_{m+k}$, respectively.   In addition, if $V$ and $V\oplus W$ carry $G-\Spin^c$-structures, then there is a unique $G-\Spin^c$-structure on $W$ whose direct sum, as above, with the $G-\Spin^c$-structure on $V$ is the given $G-\Spin^c$-structure on the direct sum.  This is the \emph{two out of three principle} for $\Spin^c$-structures.

If  $M$ is a smooth $G$-manifold, then a \emph{$G$-$\Spin^c$-structure} on $M$ is a $G$-$\Spin^c$-structure on its tangent bundle, and a \emph{$G$-$\Spin^c$-manifold} is a smooth $G$-manifold together with a given $G$-$\Spin^c$-structure.

\begin{definition}
Let $X$ be any $G$-space.
An \emph{equivariant $K$-cycle}  for $X$ is a triple $(M,E,f)$ consisting of:
\begin{enumerate}[\rm (a)]
\item A $G$-compact, proper  $G$-$\Spin^c$-manifold $M$ without boundary.\footnote{The manifold $M$ need not be connected.  Moreover different connected components of $M$ may have different dimensions.}
\item 
A smooth    complex $G$-vector bundle $E$ over  $M$.
\item A continuous and  $G$-equivariant  map $f\colon M \to X$. 
\end{enumerate}
\end{definition}

The geometric equivariant $K$-homology groups $K^G_*(X)$ will be obtained by placing a certain equivalence relation on  the class of all equivariant $K$-cycles.  Before describing it, we give constructions at the level of cycles that will give the arithmetic structure of the groups $K_*^G(X)$.

 If $(M,E,f)$ and $(M',E',f')$ are two equivariant $K$-cycles for $X$, then their \emph{disjoint union} is the equivariant $K$-cycle   $(M\sqcup M', E\sqcup E', f\sqcup f')$.    The operation of disjoint union will give addition.

 Let $V$ be a $G$-$\Spin^c$-vector bundle with a $G$-$\Spin^c$-structure $\varphi\colon Q \to P$.   Fix an orientation-reversing isometry of $\R^n$. Since it preserves the inner product, $\tau$ induces an automorphism of $\Cliff(n)$, and hence of $\Spin^c(n)$, that we shall also denote by $\tau$.  Consider the map $\varphi_\tau\colon Q_\tau \to P$, where:
 \begin{enumerate}[\rm (a)]

 \item    $Q_\tau$ is equal to $Q$ as a $G$-manifold, but has the twisted action 
 $
 q\cdot_\tau u = q\cdot \tau (u)$  of the group $\Spin^c(n)$.
 \item  {$\phi_\tau (q) = \phi(q)\tau$}.
 \end{enumerate}
It  defines the \emph{opposite} $G$-$\Spin^c$-vector bundle $-V$.
 Applying this to manifolds, we define 
the  \emph{opposite} of an equivariant $K$-cycle $(M,E,f)$ to be the equivariant $K$-cycle $(-M, E, f)$.  This will give the operation of additive inverse in the geometric groups $K^G_*(X)$.

If $M$ is a $\Spin^c$-$G$-manifold, then its  boundary $\partial M$ inherits a $\Spin^c$-$G$-structure.  This is obtained  from the pullback diagram 
$$
\xymatrix{ \Spin^c(n{-}1)\ar[d]\ar[r]&  \Spin^c(n)\ar[d]\\
GL(n{-}1,\R)\ar[r] & GL(n,\R)
}
$$
associated to the lower-right-corner embedding of $GL(n{-}1,\R)$ into $GL(n,\R)$ and the inclusion of $\Cliff(n{-}1)$ into $\Cliff(n)$ that maps the generators $e_k$ to $e_{k+1}$.    Using the outward-pointing normal first convention, the bundle of  frames  for  $T\partial M$ maps to  the  restriction to the boundary of the bundle of  frames  for $TM$.   A pullback construction gives the required reduction to $\Spin^c(n{-}1)$.

\begin{definition}
\label{boundary-spinc-def}
An equivariant $K$-cycle for $X$ is a \emph{boundary} if there is a $G$-compact, proper $G$-$\Spin^c$-manifold $W$ with boundary,  a smooth, Hermitian equivariant vector 
bundle $E$ over $W$ and a continuous equivariant map $f \colon W  \to
X$ such that the 
given cycle is isomorphic to $(\partial W, E\vert_{\partial W}, f\vert_{\partial W})$. 
Two equivariant $K$-cycles for $X$, $(M_1,E_1,f_1)$ and  $(M_2,E_2,f_2)$ are \emph{bordant} if the disjoint union of one with the opposite of the other is a boundary.
\end{definition}

The most subtle aspect of the equivalence relation on equivariant $K$-cycles that defines geometric $K$-homology involves certain sphere bundles over $\Spin^c$-manifolds.  To describe it we begin by considering a single sphere. 

View $S^{n-1} $ as the boundary of the unit ball in $\R^{n}$.  The frame bundle for   $\R^{n}$ can of course be identified with $\R^{n}\times GL(n,\R)$ since the columns of any invertible matrix constitute a frame for $\R^{n}$.  We can therefore equip $\R^{n}$ with the \emph{trivial} $\Spin^c$-structure $\R^{n}\times \Spin^c(n)$.  

According to the prescription given prior to Definition~\ref{boundary-spinc-def}, the associated $\Spin^c$-structure on the  sphere $S^{n{-}1}$ is given by the right  principal  $\Spin^c(n{-}1)$-bundle $Q$ whose fiber   at $v\in S^{n-1}$ is the space of all elements $u\in \Spin^c(n)$ whose image in $SO(n)$ is a matrix with first column equal to $v$.  Observe that $Q$ is  $\Spin^c(n)$-equivariant   for the  left  action of $\Spin^c(n)$ on the sphere given by the projection to $SO(n)$.

Let us now assume that $n=2k+1$.  We are going to fix a certain $\Spin^c(n)$-equivariant complex vector bundle $F$ on $S^{2k}$.  The key property of $F$ is that the $\Spin^c(n)$-equivariant index of the Dirac operator  (discussed in the next section) coupled to $F$ is equal to the rank-one trivial representation of $\Spin^c(n)$.  An explicit calculation, given in \cite{MR2330153}, shows that the dual of the positive part of the spinor bundle for $S^{n-1}$ has the required property.  It follows easily from the Bott periodicity theorem that $F$ is essentially unique (up to addition of trivial bundles, any two $F$ are isomorphic).  For what follows, any choice of $F$ will do.  The bundle $F$ and  the trivial line bundle together generate $K(S^{2k})$, and for that reason we shall call it the \emph{Bott generator}.

Following these preliminaries, we can describe the ``vector bundle modification'' step in the equivalence relation defining geometric $K$-homology. 

Let  $V$  be a   $G$-$\Spin^c$-vector bundle of rank $2k$ over a $G$-$\Spin^c$-manifold $M$ and denote by $\widehat M$ the sphere bundle\footnote{Strictly speaking, to form the sphere bundle we need a metric on $V$ and so a specific choice of  principal  bundle $Q$ within its homotopy class.  Of course, any two sphere bundles will be bordant.}   of the direct sum vector bundle $\R\oplus V$.
The manifold $\widehat M$ may be described as the fiber bundle
$$
\widehat M = Q\times _{\Spin^c(2k+1)} S^{2k} ,
$$
where $Q$ is the principal $G$-$\Spin^c(2k{+}1)$-bundle associated to $\R\oplus V$.
Its tangent bundle is isomorphic to the pullback of the tangent bundle of $M$, direct sum the fiberwise tangent bundle
$Q\times _{\Spin^c(2k+1)} TS^{2k} $.  Both carry natural $G$-$\Spin^c$-structures, and so $\widehat M$ is a $G$-$\Spin^c$-manifold.

Form  the $G$-equivariant complex vector bundle 
$$
  Q \times _{\Spin^c(2k+1)} F ,
$$
from the Bott generator discussed above.  We shall use the same symbol $F$ for this bundle over $\widehat M$.

 \begin{definition}
Let $(M,E,f)$ be  an equivariant $K$-cycle  and let $V$ be a rank $2k$
$G$-$\Spin^c$-vector bundle over $M$.   The  \emph{modification} of $(M,E,f)$
associated to $V$ is the equivariant  $K$-cycle 
$$ (M,E,f)\hat{\phantom{tt}} = \bigl(\widehat M,  F \otimes \pi^*(E) , f\circ \pi\bigr),$$ where:
\begin{enumerate}[\rm (a)]
\item  $\widehat M$ is the total space of the sphere bundle of   $\R \oplus V$, equipped with the $G$-$\Spin^c$-structure described above;  
\item $\pi$ is the projection from $ \widehat M $ onto $M$; and 
\item  $F$ is the $G$-equivariant complex vector bundle on $\widehat M$   described above.
\end{enumerate}
\end{definition}

We are now ready to define the  geometric equivariant $K$-homology groups.

\begin{definition}
\label{k-geom-def}
Denote by $K^{G}(X)$ the set of equivalence classes of equivariant $K$-cycles over $X$, for the equivalence relation generated by the following three elementary relations:
\begin{enumerate}[\rm (a)]
\item If $(M,E_1,f)$ and $(M,E_2,f)$ are two equivariant  $K$-cycles with the same proper, $G$-compact $G$-$\Spin^c$-manifold $M$ and same map $f\colon M\to X$, then 
$$ (M\sqcup M ,E_1\sqcup E_2 ,f\sqcup f ) \sim  (M, E_1\oplus E_2, f).
$$
\item 
If $(M_1,E_1,f_1)$ and $(M_2,E_2,f_2)$ are bordant equivariant $K$-cycles then 
$$(M_1,E_1,f_1) \sim (M_2,E_2,f_2).$$
\item If $(M,E,f)$ is an equivariant  $K$-cycle, if $V$ is an even-rank $G$-$\Spin^c$-vector bundle over $M$,  and if $(M,E,f)\hat{\phantom{tt}}$ is the modification of $(M,E,f)$ associated to $V$,   then     
$$ 
(M,E,f)\sim  (M,E,f)\hat{\phantom{tt}}.$$ 
\end{enumerate}
\end{definition}

The set $K^{\geomG}(X)$ is  an abelian group with addition given by disjoint union.  Denote by $K^{\geomG}_{\ev}(X)$ and $K^{\geomG }_{\odd}(X)$ the subgroups of $K^{\geomG}(X)$ composed of equivalence classes of equivariant $K$-cycles $(M,E,f)$ for which every connected component of $M$ is even-dimensional or odd-dimensional, respectively.  Then $K^{\geomG}(X) \cong K^{\geomG}_{\ev}(X)\oplus K^{\geomG }_{\odd}(X)$.

\section{Equivariant Kasparov Theory}
\label{kk-review-sec}

In this section we shall define a natural transformation from   geometric equivariant $K$-homology to Kasparov's equivariant $K$-homology.  Once again, this  is a straightforward extension  to the equivariant context of the Baum-Douglas theory that was reviewed in detail   already  in the paper \cite{MR2330153}.  Therefore we shall be brief.

Fix a \emph{second countable} $G$-compact proper $G$-space $X$, for example a $G$-finite proper $G$-$CW$-complex.   The second countability assumption is made for consistency with Kasparov's theory, which applies to second countable locally compact spaces, or separable $C^*$-algebras.

We shall denote by  $KK_n^G(C_0(X),\C)$ the Kasparov group 
$KK^G \bigl (C_0(X), \Cliff(n)\bigr)$ (the action of $G$ on $\Cliff(n)$ is trivial).  See \cite[Section 2]{MR918241}. There are canonical isomorphisms 
$$
KK_n^G(C_0(X),\C) \cong KK_{n+2}^G(C_0(X),\C)
$$
coming from the periodicity of Clifford algebras. Compare  \cite{MR2330153}.  As a result we may form the $2$-periodic groups $KK_{\ev/\odd}^G(C_0(X),\C)$.

The natural transformation 
$$
\mu\colon K^{G}_{\ev/\odd} (X)\longrightarrow 
KK^G_{\ev/\odd} (C_0(X),\C)
$$
 into Kasparov theory is defined by associating to an equivariant $K$-cycle $(M,E,f)$ a Dirac operator, and then constructing from the Dirac operator  a cycle for Kasparov's analytic $K$-homology group.

The vector space $\Cliff(n)$ carries a natural inner product in which the monomials $e_{i_1}\cdots e_{i_k}$ form an orthonormal basis.  If $M$ is a $G$-compact, proper $G$-$\Spin^c$-manifol, and if $Q$ is a   lifting to $\Spin^c(n)$  of the frame bundle of $M$, then the  $\Z/2\Z$-graded Hermitian vector bundle 
$$ S = Q\times _{\Spin(n)} \Cliff (n)  $$
that is formed using the left multiplication action of $\Spin^c(n)$ on $\Cliff (n)$ carries a right action of the algebra $\Cliff(n)$ and a commuting left action of $TM$ as odd-graded skew-adjoint endomorphisms such that $v^2 = -\|v\|^2 I$.  This is called the action of $TM$ on the \emph{spinor bundle} $S$ by \emph{Clifford multiplication}.  

\begin{remark}
There are other versions of the spinor bundle that do not carry a right Clifford algebra action.  The bundle used here has the advantage of allowing a uniform treatment of both even and odd-rank bundles $V$.  In addition the real case may be treated similarly (although we shall not consider it in this paper).
\end{remark}

 \begin{definition}
\label{dirac-op-def}
Let $M$ be a  $G$-compact proper $G$-$\Spin^c$-manifold. Fix an associated principal $\Spin^c$-bundle over $M$, and let $S$ be the spinor bundle, as above.  Let $E$ be a smooth, Hermitian $G$-vector bundle over $M$. 
We shall call an  odd-graded, symmetric, first-order  {$G$-equivariant} linear partial differential operator $D$ acting on the sections of  $S\otimes E$ a \emph{Dirac operator} if it  commutes with the right Clifford algebra action on the spinor bundle and if 
$$
[ D, f] u = \operatorname{grad} (f)\cdot u,
$$
for every smooth function $f$ on $M$ and every section $u$ of $S\otimes E$, where $ \operatorname{grad} (f)\cdot u$ denotes Clifford multiplication on $S$ by the gradient of $f$.  \end{definition}

Dirac operators in this sense always exist, and basic   PDE theory gives the following result: 

\begin{proposition} 
\label{elliptic-prop}
The Dirac operator $D$, considered as an unbounded operator  on $L^2(M, S\otimes E)$   with domain the smooth compactly supported sections  is essentially self-adjoint.
 The bounded {$G$-equivariant} Hilbert space operator 
 $F = D (I+D^2)^{-1/2}$ 
 commutes, modulo compact operators, with multiplication operators from $C_0(M)$.  Moreover the product of $I-F^2$ with any multiplication operator from $C_0(M)$ is a compact operator. \qed
\end{proposition}

Now the Hilbert space $L^2(M, S\otimes E)$ carries a right action of $\Cliff(n)$ that commutes with $D$ and the action of $C_0(M)$.  It also carries a unique $\Cliff(n)$-valued inner product $\langle\,\, , \,\, \rangle_{\Cliff}$ such that 
 $$\langle s_1,s_2\rangle  =  \tau \bigl ( \langle s_1,s_2\rangle_{\Cliff} \bigr )
 $$
 where on the left is the $L^2$-inner product, and on the right is the state $\tau$ on $\Cliff(n)$ that maps all nontrivial monomials $e_{i_1}\cdots e_{i_p}$ to zero.  Using it we place a Hilbert $\Cliff(n)$-module structure on $L^2 (M,S\otimes E)$.
 
Proposition~\ref{elliptic-prop} implies that the operator $F= D (I+D^2)^{-1/2}$, viewed as an operator on the Hilbert $\Cliff(n)$-module $L^2(M,S\otimes E)$, yields a cycle for Kasparov's equivariant $KK$-group $KK^G(C_0(M),\Cliff(n))$ (see \cite[Definition 2.2]{MR918241}).

\begin{definition} We shall denote by $[M,E]\in KK_n^G(C_0(M),\C)$ the $KK$-class of the operator  $ F= D (I+D^2)^{-1/2}$.  \end{definition} 
 
 The first main theorem concerning the classes $[M,E]$ is as follows:

\begin{theorem} 
\label{mu-well-defined}
The correspondence that associates to each  equivariant $K$-cycle $(M,E,f)$ the $KK$-class
$$ f_* [M,E]\in KK_{\ev/\odd} ^G (C_0(X),\C)$$ gives a well-defined homomorphism $$
\mu\colon K^G _{\ev/\odd}(X)\longrightarrow KK_{\ev/\odd}^G (C_0(X),\C) .
$$
\end{theorem}

The   non-equivariant case of the theorem is proved in   \cite{MR2330153}.  The proof for the equivariant case is exactly the same and therefore  will not be repeated. 

Our   aim in this paper is to prove the second main theorem concerning the classes $[M,E]$.

\begin{theorem}
\label{main-theorem}
If $X$ is any proper, $G$-finite $G$-$CW$-complex, then the index map 
$$
\mu\colon K^G _{\ev/\odd}(X)\longrightarrow KK_{\ev/\odd}^G (C_0(X),\C)
$$
is an isomorphism.
\end{theorem}

The non-equivariant version of the theorem is due to Baum and Douglas, and is proved in detail in \cite{MR2330153}.  Although the proof of the equivariant result is   the same  in outline,   new issues must also be resolved   having to do with the properties of equivariant vector bundles on $G$-compact proper $G$-spaces.  These we shall consider next.

\section{Equivariant Vector Bundles}

\label{vb-sec}
Throughout this section we shall use the term \emph{$G$-bundle} as an abbreviation for  \emph{$G$-equivariant complex vector bundle.}
We shall review 
the basic theory of   $G$-bundles over $G$-compact proper $G$-spaces, mostly
as  worked out by L\"uck and Oliver in \cite{MR1838997}.  In the next section we shall recast their results in the language of $C^*$-algebra $K$-theory.

\begin{theorem} 
\label{L-O-thm}
Let $X$ be a $G$-compact, proper $G$-space. There is a   $G$-bundle 
$E$ over $X$ such that for every $x \in X$, the fiber $E_x$  is
{contained in} a multiple of the regular representation of the isotropy group $G_x$.  \qed
\end{theorem}

\begin{proof}
This is proved for $G$-finite proper $G$-$CW$-complexes in  \cite[Corollary~2.8]{MR1838997}.
That result extends to  more general  $X$     by pulling back   along the map supplied by Theorem~\ref{ebar-numerable-thm}.  
\end{proof}

\begin{corollary}[{Compare \cite[Lemma~3.8]{MR1838997}}] 
\label{dir-summand-lemma}
Let $Z$ be a $G$-compact proper $G$-space and let $X$ be a closed, $G$-invariant  subset of $Z$.  If $F$ is any $G$-bundle on $X$, then there is a $G$-bundle $E$ on $Z$ such that $F$ embeds as a summand of  $E\vert_{X}$.
\end{corollary}

\begin{proof}
Fix a $G$-bundle $E$ on $Z$, as in Theorem~\ref{L-O-thm}.   There  are $G$-invariant open subsets $U_1,\dots, U_n$ of $X$ such that:
\begin{enumerate}[\rm (a)]
\item The sets cover $X$.
\item For each $j$ there  is a finite subgroup of $F_j\subseteq G$ and 
an equivariant map $\pi_j\colon  U_j\to G/F_j$.
 \item   $F\vert _{U_j}$ is isomorphic to a bundle pulled back along $\pi_j$.
\item  $E\vert_{U_j}$ is also  isomorphic to a bundle pulled back along  $\pi_j$.
\end{enumerate}
Replacing $E$ by a direct sum $E\oplus \cdots \oplus E$, if necessary, we find that $F\vert _{U_j}$ may be embedded as a summand of $E\vert _{U_j}$, for every $j$.  Making a second replacement of $E$ by an $n$-fold direct sum $E\oplus \cdots \oplus E$ and using a standard partition of unity argument, we may now embed $F$ into $E$, as required.
\end{proof}

More generally, if $f\colon X\to Z$ is a map between $G$-compact proper $G$-spaces, and if $F$ is a $G$-bundle on $X$, then the same argument shows that $F$ is isomorphic to a summand of the pullback along $f$ of some $G$-bundle  on $Z$.
 
\begin{definition} 
If $S$ is any set, then denote by $\C[S]$ the free vector space on the set $S$,   equipped with the standard inner product in which the elements of $S$ are orthonormal.  If $S$ is equipped with an action of $G$, then we shall consider $\C[S]$ to be equipped with the corresponding permutation action of $G$.
\end{definition}

We are interested primarily in the case where $S=G$, which we shall view as equipped with the usual left translation action of $G$.

\begin{definition}
\label{standard-def}
A \emph{standard $G$-bundle} on a $G$-compact proper $G$-space $X$ is a $G$-invariant subset $E$  of  $X\Times \C[G]$ with the property  that for every compact subset $K\subseteq X$ there is a \emph{finite} subset $S\subseteq G$ such that the intersection  of $E$ with   $K\Times \C[G]$ is a (nonequivariant) complex vector subbundle of $K\times \C[S]$.
\end{definition}
 
\begin{remark}
We require that the restriction of $E$ to $K$, as above,  be a topological vector subbundle of the finite-dimensional trivial bundle $K\Times \C[S]$.  This fixes the topology on $E$ and determines a   $G$-bundle structure.
\end{remark}

It follows from a standard partition of unity  argument  that every $G$-bundle on $X$ is isomorphic to a standard $G$-bundle.  
We are going to prove the following result, which gives the set of standard $G$-bundles a useful directed set structure.

\begin{theorem}  
\label{proper-bundle-thm}
Any two standard $G$-bundles   are subbundles of a common third.  Moreover the union of all standard $G$-bundles is $X\Times\C[G]$.\end{theorem}

\begin{remark}
\label{modification-remark}
In Section~\ref{proof-sec} we shall modify Definition~\ref{standard-def} very slightly by replacing $G$ with a countable disjoint union $G_\infty = G\sqcup G \sqcup \dots $ (thought of as a left $G$-set).    Theorem~\ref{proper-bundle-thm} remains true, with the same proof.
\end{remark}

Since the theorem is obvious if $G$ is finite, we shall assume $G$ is infinite  until the proof of Theorem~\ref{proper-bundle-thm} is concluded.

\begin{lemma} 
\label{standard-bundle-lemma}
If  $S$ is a finite subset of $G$, if $K$ is a compact subset of $X$, and if $E $ is any $G$-bundle over $X$, then there is  a standard $G$-bundle $E_1$ that is isomorphic to $E$ and whose restriction to $K$ is orthogonal to $K\Times \C[S]$. 
\end{lemma}

\begin{proof}
Let $E_1$ be any standard bundle that is isomorphic to the equivariant vector bundle $E$. If $g\in G$, then set   
$$E_1\cdot g = \{\,  (x,e\cdot g )\, :\, e\in E_{1,x}\,\}$$
is also standard $G$-bundle.  Here, in forming the vectors $e\cdot g$ we are
using the \emph{right} translation  action of $G$ on itself and hence on
$\C[G]$.  The bundle ${ E_1\cdot g}$ is isomorphic to $E_1 $, and hence to $E$.    If we enlarge $S$, if necessary, so that $E_1\vert _K \subseteq K\times \C[S]$, and if we choose $g\in G$ so that    $S\cap  Sg  = \emptyset$, then       $(E_1\cdot g)\vert _K $ 
 is orthogonal to $K\Times \C[S]$, as required. \end{proof}

\begin{lemma} 
\label{simple-orth-lemma}
Let $E$ be a $G$-bundle on $X$ and let $E_2$ be a standard $G$-bundle.  There is a standard $G$-bundle that is isomorphic to $E$ and orthogonal to $E_2$.  \end{lemma}

\begin{proof}  
Let $K$ be a compact subset of $X$ whose $G$-saturation is $X$, and let $S$ be a finite subset of $G$ such that  $E_2\vert _K\subseteq K\times \C[S]$.  Now apply the previous lemma.\end{proof}

\begin{lemma}
\label{technical-L-O-lemma}
Let $U$ be a $G$-invariant open subset of $X$, and let $Y$ and $Z$  be  $G$-invariant closed  subsets of $X$ such that 
$$ Z \subseteq U \subseteq Y \subseteq X.$$
Let $F$ be a  standard $G$-bundle over $Y$.  There is a standard $G$-bundle  $E_1$ over $X$ such that $F\vert _Z \subseteq E_1\vert _Z$.  Moreover given a standard $G$-bundle $E_2$ over $X$ such that $E_2\vert_{Y}$ is orthogonal to $F$, the standard $G$-bundle $E_1$ may be chosen to be orthogonal to $E_2$.
\end{lemma}

\begin{proof}
According to Corollary~\ref{dir-summand-lemma}, there is a $G$-bundle $E$ over $X$ such
that $F$ embeds in $E\vert_Y$.  Any complement of {the image
  of} $F$ in $E\vert_Y$ may be embedded as a standard $G$-bundle $F'$ on $Y$
that is orthogonal to ${F\oplus E_2\vert_Y}$.

{We can choose an isomorphism 
 $\Phi \colon E\vert_Y\to  F\oplus F'$.} 
Next,  there is an embedding  $\Psi$   of $E$ as a standard $G$-bundle on  $X$
such that {$\Psi[E]$ is orthogonal to $E_2$ and
  $\Psi[E]\vert_Y$}  is orthogonal to $F$ (by a slight elaboration of
Lemma~\ref{simple-orth-lemma}).  If we choose a $G$-invariant scalar function
$\varphi$ on $X$ such that $\varphi= 1$ on $Z$ and $\varphi = 0$ outside of
$U$, and if we set $\psi = 1-\varphi$, then  $E_1 = (\varphi \Phi + \psi
\Psi)[E] $ has the required properties.\end{proof}

\begin{lemma}\label{lem:compact_emb_standard}
Let $K$ be any compact subset of $X$, and  let    $S$ be a finite subset of $G$.  There is a standard $G$-bundle  that contains $K\times \C[S]$.
\end{lemma}

\begin{proof}
The compact set $K$ may be written as a finite union of compact sets 
$$
K = K_1 \cup \cdots \cup K_n
$$
where  each $K_j$ is included in a $G$-invariant open set that maps equivariantly onto some proper coset space $G/H_j$, in such a way that $K_j$ maps to the identity coset.  We shall  use induction on $n$.

 Let $E_2$ be a standard $G$-bundle that contains the set 
$$
\bigl ( K_1 \cup \cdots \cup K_{n-1}\bigr ) \times \C [S].
$$
There is a $G$-compact subset $Y\subseteq X$ that contains a $G$-invariant neighborhood $U$ of $K_n$ and over which there is a standard $G$-bundle $L$ such that 
$$K_n\!\times\! \C[S]\subseteq L\vert _{K_n} \quad \text{and} 
\quad E_2 \vert_{Y}\subseteq L.$$
Indeed we may choose $Y$ so that it   maps equivariantly to $G/H_n$, and if $Y_n$ is the inverse image of the identity coset, then we may form
$$L = \bigcup_{g\in G}\,\, gY_n\times {\C[S_ng]},$$
where $S_n$ is a sufficiently large finite and right $H_n$-invariant subset of $G$.

Now apply the previous lemma to the standard $G$-bundle  $F=L\ominus E_2\vert_{Y}$ (the orthogonal complement of ${E_2\vert_{Y}}$ in $L$) to obtain a standard $G$-bundle  $E_1$ on $X$ such that $E_1$ is orthogonal to $E_2$ and 
$$
L\vert_{K_n}\ominus E_2\vert_{K_n}\subseteq E_1\vert_{K_n} .
$$
The standard $G$-bundle $E_1\oplus E_2$ then contains $K\Times \C[S]$, as required.\end{proof}

\begin{proof}[Proof of Theorem~\ref{proper-bundle-thm}]
  Since there is a standard $G$-bundle that contains any given $K\times \C[S]$, it is clear that the union of all standard $G$-bundles is $X\times \C[G]$.  
Let $E_1$ and $E_2$ be  standard $G$-bundles on $X$. 
Choose a compact set $K$ whose $G$-translates cover $X$ and choose a finite set $S\subseteq G$ such that the     $E_1\vert_{K}, E_2\vert_{K}\subseteq K\Times \C[S]$.  
If $E$ is a standard $G$-bundle containing $K\Times \C[S]$, then it contains  $E_1$ and $E_2$.
\end{proof}

\section{C*-Algebras and Equivariant K-Theory}
\label{cstar-sec}

\begin{definition}
If $S$ is any set, then denote by $M[S]$ the $*$-algebra of complex matrices $[T_{s_1s_2}]$ with rows and columns parametrized by the set $S$, all but finitely many of whose entries are zero. 
\end{definition}

We shall be interested in the case where $S=G$.  In this case the group $G$ acts on $M[G]$ by automorphisms via the formula 
$(g\cdot T)_{g_1,g_2} = T_{g^{-1}g_1,g^{-1}g_2}$.

\begin{definition}
Let $X$ be a $G$-compact proper $G$-space.
 Let us call a function  $F\colon X \to M[G]$   \emph{standard} if its matrix element functions  $$F_{g_1,g_2}\colon x\mapsto F(x)_{g_1,g_2}$$  are continuous and compactly supported, and if  for every compact subset $K$ of $X$   all but finitely many of them vanish outside of $K$. 
We shall denote by $\mathcal C (X,G)$ the $*$-algebra of all standard, $G$-equi\-variant functions from $X$ to $M[G]$.
\end{definition}

Note   that if  $P$ is a projection in the $*$-algebra $\mathcal C (X,G)$, then the range of $P$ (that is the bundle over $X$ whose fiber of $x\in X$ is the range of the projection operator $P(x)$ in $\C[G]$) is a standard $G$-bundle in the sense of Section~\ref{vb-sec}.  In fact every standard $G$-bundle is obtained in this way, which explains our interest in $\mathcal C (X,G)$.
In fact we are even more interested in the following $C^*$-algebra completion of $\mathcal C (X,G)$.

\begin{definition}
 Let $X$ be a $G$-compact proper $G$-space.
Denote by $C^*(X,G)$ the $C^*$-algebra of $G$-equivariant, continuous functions from $X$ into the compact operators on $\ell^2(G)$. 
\end{definition}

 \begin{remark}
The $C^*$-algebra $C^*(X,G)$ is isomorphic to the crossed product $C^*$-algebra $C_0(X)\rtimes G$.  
 If $E^{g_1,g_2}$ denotes the matrix with $1$ in entry $(g_1,g_2)$ and zero in every other entry, then the formula $$f\,[g]
\mapsto \sum_h h(f)E^{h,hg}$$ gives an isomorphism from $C_0(X)\rtimes G$ to $C^*(X,G)$, and the formula
$$F\mapsto  \sum_{g\in G} F_{e,g}\, [g] $$
gives its inverse.   Since the action of $G$ on $X$ is proper, the maximal and reduced crossed products are equal. Indeed there is a unique $C^*$-algebra completion of the $*$-algebra $\mathcal C(X,G)$.\end{remark}

 \begin{lemma}
 \label{bijections-lemma}
 Assume that $G$ is infinite and $X$ is a $G$-compact proper $G$-space. The correspondence between projections in $\mathcal C (X,G)$ and their ranges induces bijections among the following sets:
\begin{enumerate}[\rm (a)]
\item Equivalence classes of projections in $\mathcal C (X,G)$.
\item Equivalence classes of projections in $C^*(X,G)$.
\item Isomorphism classes of standard $G$-bundles on $X$.
\item Isomorphism classes of hermitian $G$-bundles on $X$.
\end{enumerate}
If $X$ is a $G$-compact proper $G$-manifold, then there is in addition a bijection with 

\smallskip

\noindent{\rm (e)}\,\, Isomorphism classes of  smooth hermitian $G$-bundles on $X$.
\end{lemma}

\begin{proof}
Recall that two projections $P$ and $Q$ in a $*$-algebra are equivalent if and only if there is an element $U$ such that $U^*U=P$ and $UU^* = Q$.  The inclusion of 
 $\mathcal C (X,G)$ into $C^*(X,G)$  is  a simple example of a \emph{holomorphically closed subalgebra}, and as a result the  inclusion  induces a bijection between the sets in (a) and (b).  Compare \cite[Sections 3 and 4]{MR1656031}. The sets in (a) and (c) are in bijective correspondence virtually by definition.  The sets in (c) and (d) are in bijection thanks to Lemma~\ref{standard-bundle-lemma}, which in particular shows that every $G$-bundle is isomorphic to a standard bundle, if $G$ is infinite.  
 
 If $X$ is a manifold, then the inclusion of the smooth functions in $\mathcal C (X,G)$ into $C^*(X,G)$ is also a holomorphically closed subalgebra, and this gives the final part of the lemma since equivalence classes projections in the algebra of smooth functions correspond to isomorphism classes of (the obvious concept of) smooth standard $G$-bundles. 
 \end{proof}

\begin{remark}
If $G$ is finite, then the lemma remains true if $\mathcal C (X,G)$ and $C^*(X,G)$ are replaced by direct limits of matrix algebras over themselves.
\end{remark}

\begin{theorem}
The $C^*$-algebra $C^*(X,G)$ has an approximate identity consisting of projections.
\end{theorem}

\begin{proof}
We claim that for  every finite set of elements $F_1,\dots , F_n$ in $\mathcal C (X,G)$ there is a projection $P$ in $\mathcal C (X,G)$ such that 
$$
F_j = PF_j = F_j P
$$
for all $j=1,\dots n$.
Indeed  the orthogonal projection  onto any standard $G$-bundle is a
projection  in   $\mathcal C (X,G)$ (and conversely). But
  on any compact set $K$, all but finitely many matrix coefficients of the
  $F_j\vert_{K}$ are zero. So by Lemma \ref{lem:compact_emb_standard} the images
  of $F_j$ and $F_j^*$  are contained in a (standard) $G$-bundle. The
  theorem follows  since $\mathcal C (X,G)$ is dense in $C^*(X,G)$. 
\end{proof}

\begin{theorem}
\label{bijections-thm} Let $X$ be a $G$-compact proper $G$-space.
The bijections in Lemma~\ref{bijections-lemma} determine a natural  isomorphism between the Grothendieck group of $G$-bundles on $X$  and  the $K_0$-group of the $C^*$-algebra $C^*(X,G)$.
\end{theorem}
 
 \begin{proof}
 If $A$ is any $C^*$-algebra with an approximate unit consisting of projections, then the natural map from the Grothendieck group of projections in matrix algebras over $A$ into $K_0(A)$ is an isomorphism.  So the theorem follows from the previous result.
 \end{proof}

\begin{definition}
If  $X$ is a $G$-compact proper $G$-space, and $j\in \Z/2\Z$, then denote by $K_G^j(X)$ the $K_j$-group of the $C^*$-algebra $C^*(X,G)$.  If $Y$ is a $G$-invariant closed subset of $X$, then denote by $K_G^j(X,Y)$ the $K_0$-group of the ideal in $C^*(X,G)$ consisting of functions that vanish on $Y$.
\end{definition}

 By the above, $K^0_G(X)$ is the Grothendieck group of isomorphism classes of $G$-bundles on $X$.  

The relative groups $K^j_G(X,Y)$ satisfy excision   (of the strongest possible type, that $K^j_G(X,Y)$ depends only on $X\setminus Y$).  Elementary $K$-theory provides functorial coboundary maps $$\partial \colon K_G^{j}(Y)\longrightarrow K_G^{j+1}(X,Y),$$
and these give the groups $K^j_G(X,Y)$ the structure of   a $\Z/2\Z$-graded cohomology theory on $G$-compact proper $G$-spaces, in the sense that that they  fit into functorial long exact sequences 
$$
 \xymatrix{
  \dots 
  \ar[r] &
  K_G^{j}(X)
  \ar[r] & 
   K_G^{j}(Y)
   \ar[r] &
    K_G^{j+1}(X,Y)
    \ar[r]& 
     K_G^{j+1}(X)
    \ar[r]& 
    \dots
 }  
$$
Although we have accessed this fact using $C^*$-algebra $K$-theory, this is also the main result of \cite{MR1838997}.

We conclude by reviewing the Gysin  maps in equivariant $K$-theory that we shall need in the next section.  
Let $E$ be a  complex $G$-bundle over a $G$-compact proper $G$-manifold $M$.  As we noted earlier, $E$ carries a canonical $\Spin^c$-structure. Form the sphere bundle $\widehat M$ of the real bundle $\R\oplus E$, as in Section~\ref{geometric-sec}.  The manifold  $M$ is equivariantly embedded as a retract in $\widehat M$ using the section $$M\ni m\mapsto (1,0)\in \R\oplus E_m,$$ and associated to the embedding is a short exact sequence of $K$-theory groups 
$$
0\longrightarrow K^0_G (\widehat M,M)\longrightarrow K^0_G(\widehat M) \longrightarrow K^0_G (M)\longrightarrow 0.
$$
Let $F$ be the complex $G$-bundle  over $\widehat M$ that we defined in
Section~\ref{geometric-sec}, and denote by $F_0$ the complex $G$-bundle
obtained by restricting $F$ to $M$, then pulling back the restriction to $
\widehat M$ using the projection from $\widehat M$  down to $M$.  Because of the above exact sequence the
difference  $[F] - [F_0]$   defines an element of $K_G^0(\widehat M, M)$.  The relative group
is a module over $K_G^0(M)$ {via pullback and tensor product}, and multiplication against $[F]-[F_0]$ gives the \emph{Thom homomorphism}
$$
  K_G^0(M)\longrightarrow K_G^0(\widehat M, M). $$
The complement $\widehat M \setminus M$ identifies with the total space of $E$ via the map
$$
E\ni e \mapsto \tfrac{1}{1+ \|e\|^2 } \left ( 1, e\right )\in \R\oplus E.
$$
  If $\iota \colon  M \to N$ is an embedding of $G$-compact proper $G$-manifolds, and if the normal bundle to the embedding is identified with $E$, then the \emph{Gysin map} $$\iota_!\colon K^0_G(M )\to K^0_G (N)$$
 is the composition 
$$
\xymatrix{
K^0_G(M)\ar[r]^-{\operatorname{Thom}} & K^0_G (\widehat M,M) \ar[r]^-\cong & K^0_G(N,N') \ar[r] & K^0_G(N) ,
}
$$
where $N'\subseteq N$ is the complement of a tubular neighborhood   and the middle map is given by excision and the identification of the tubular neighborhood with $E$.  

The Gysin map is functorial for compositions of embeddings.  It is well-defined for embeddings of manifolds with boundary, as long as the embedding is transverse to the boundary of $N$, and carries the boundary of $M$ into the boundary of $N$.

\section{The Technical Theory}

 In this section we shall construct the homology groups $k^G_*(X)$ that were described in the introduction.  They are obtained as direct limits of  certain bordism groups.

\begin{definition}
\label{ZEmfld-def}
  Let $Z$ be a proper $G$-space and let $E$ be a $G$-bundle on $Z$.
A \emph{stable $(Z,E)$-manifold} is  a $G$-compact proper $G$-manifold $M$ (possibly with boundary)  together with an equivalence class of pairs $(h,\varphi)$, where:
\begin{enumerate}[\rm (a)]
\item $h\colon M \to Z$ is a continuous and $G$-equivariant map.
\item $\varphi$  is an isomorphism of topological  real $G$-bundles
$$
\varphi \colon   \R^r\oplus TM\longrightarrow \R^s\oplus h^*E    , 
$$
for some   $r,s\ge 0$. Here   $\R^r$ and $\R^s$ denote the trivial bundles of ranks $r$ and $s$ (with trivial action of  $G$ on the fibers).
\end{enumerate}
The equivalence relation is stable homotopy: $(h_0,\varphi_0)$ and $(h_1,\varphi_1)$ are equivalent if there is a homotopy 
$h\colon M\times[0,1]\to Z$ between $h_0$ and $h_1$ and an isomorphism of real $G$-bundles over $M\times[0,1]$ 
$$\varphi\colon \R^r\oplus TM \longrightarrow \R^s \oplus h^* E ,
$$
with $r\ge r_0, r_1$, that restricts to $ I_{r-r_0}\oplus \varphi_0$ and $I_{r-r_1}\oplus \varphi_1$ at the two endpoints of $[0,1]$.
\end{definition}
 
If $M$ is a stable $(Z,E)$-manifold with boundary, then its boundary may be equipped with a stable $(Z,E)$-structure by forming the composition
$$
\xymatrix{
\R^r\oplus \R\oplus T\partial M \ar[r]_\cong &
\R^r\times TM\big\vert_{\partial M} \ar[r]_-{\varphi} &  \R^s \oplus f^*E,} 
$$
in which $\R\oplus T\partial M$ is identified with  $TM\vert _{\partial M}$ by the ``exterior normal first'' convention.

 \begin{definition}
Let $X$ be a  $G$-finite proper $G$-$CW$-complex and let $Y$ be a $G$-subcomplex of $X$.   For  $j = 0,1,2,\dots$ we define $\Omega^{(Z,E)}_j(X,Y)$ to be the group of equivalence classes of triples 
$(M,a,f)$, where
\begin{enumerate}[\rm (a)]
\item $M$ is a smooth, proper, $G$-compact $G$-manifold of dimension $j$ with a stable  $(Z,E)$-structure. 
\item $a$ is a class in the group $K^0_G(M)$.
\item $f\colon M\to X$ is a continuous, $G$-equivariant map  such that    $f[\partial M]\subseteq Y$.
\end{enumerate}
The equivalence relation is the obvious notion of bordism.\end{definition}

\begin{remark}
Of course, the relation of bordism is arranged to incorporate the   classes $a\in K^0_G(M)$, so that if $(M,a,f)$ is the boundary of $(W,b,g)$, then not only do we have that $M=\partial W$ and $f=g\vert _{M}$, but also the restriction map $K^0_G(W)\to K^0_G(M)$ takes $b$ to $a$.   One approach to the concept of bordism between manifolds with boundary is reviewed in \cite[Definition~5.5]{MR2330153}.
\end{remark}

The sets $\Omega^{(Z,E)}(X,Y)$ are abelian groups. The  group operation is
given by disjoint union  and the additive inverse of $(M,a,f)$ is $(-M,a,f)$.
Here the \emph{opposite}   $-M$ is obtained  by composing the bundle
isomorphism $\varphi$ with an orientation-reversing automorphism of
$\R^s$. (See Lemma~\ref{fried-egg-lemma} for another description of the
inverse.)

The  groups $\Omega^{(Z,E)}_j(X,Y)$ constitute  a homology theory on
$G$-finite proper $G$-$CW$-complexes. Homotopy invariance follows from the
bordism relation. The boundary   maps 
$$\partial \colon \Omega_j^{(Z,E)}(X,Y)\longrightarrow \Omega^{(Z,E)}_{j-1}(Y)$$
  take $(M,a,f)$ to $(\partial M, a\vert _{\partial M}, f\vert_{\partial M})$.  They fit into  sequences 
  \begin{equation}\label{eq:LES}
 \xymatrix{
  \dots 
  \ar[r] &
  \Omega_j^{(Z,E)}(X)
  \ar[r] & 
   \Omega_j^{(Z,E)}(X,Y)
   \ar[r] &
    \Omega_{j-1}^{(Z,E)}(Y)
    \ar[r]& 
     \Omega_{j-1}^{(Z,E)}(X)
    \ar[r]& 
    \dots
 } . 
\end{equation}
whose exactness follows from direct manipulations with cycles, as follows.  A cycle $(M ,a,f)$ for $\Omega_j^{(Z,E)}(Y)$, when mapped to $\Omega^{(Z,E)}_j(X,Y)$, represents the zero class since it is the boundary of $(M\times [0,1],a,f\circ pr_M)$.  Conversely if a cycle for  $\Omega_j^{(Z,E)}(X)$ is a boundary in $\Omega^{(Z,E)}_f(X,Y)$, then that boundary is a bordism from the given cycle to a cycle for $\Omega_j^{(Z,E)}(Y)$.  So the sequence is exact at $\Omega_j^{(Z,E)}(X)$.

Exactness at $\Omega_{j-1}^{(Z,E)}(Y)$ is evident, as is the fact  that the composition of the two maps through  $\Omega_j^{(Z,E)}(X,Y)$ is zero.  
It remains to complete the proof of    exactness at $\Omega_j^{(Z,E)}(X,Y)$.  If a cycle $(M,a,f)$ for  $\Omega_j^{(Z,E)}(X,Y)$ maps to zero in $\Omega_{j-1}^{(Z,E)}(Y)$, then one can glue to $(M,a,f)$ a null bordism 
for its image  in  $\Omega_{j-1}^{(Z,E)}(Y)$ to lift $(M,a,f)$ to a cycle for 
$\Omega_j^{(Z,E)}(X)$.  The glued manifold carries and evident   stable $(Z,E)$-structure.  As for the lifting of the $K$-theory class $a$, we use the fact that  if $a_1$ and $a_2$ are equivariant $K$-theory classes on manifolds $M_1$ and
$M_2$ that restrict to a common class on the a common boundary, then there is a
$K$-theory class on $M_1\cup M_2$ that restricts to $a_1$ and $a_2$.  This follows from
the Mayer-Vietoris sequence for equivariant $K$-theory, and therefore from the
results of L\"uck and Oliver reviewed in Sections~\ref{vb-sec} and
\ref{cstar-sec}.

A map of pairs of $G$-$CW$-complexes ${\phi\colon }(X_1,Y_1) 
\to (X_2,Y_2)$ that is a homeomorphism from $X_1\setminus Y_1  $ to
$X_2\setminus Y_2$ induces an isomorphism on relative groups  as follows. The
open cells in $X_1\setminus Y_1$ and $X_2\setminus Y_2$ correspond to each
other. By induction on the dimension of cells we can construct open
$G$-invariant neighborhoods $U_1$ and $U_2$ of $Y_1$ and $Y_2$, respectively,
with $G$-equivariant deformation retractions $U_i\to Y_i$. Using the long
exact sequence \eqref{eq:LES} we can replace $Y_i$ with $U_i$. Given a cycle
$(M,a,f)$ for $\Omega_j^{(Z,E)}(X_2,U_2)$ we can, using a $G$-invariant collar of $M$,
replace it by a cycle which omits $Y_2$ and therefore lifts to
$(X_1,U_1)$. Similarly, we can achieve that a bordism for $(X_2,U_2)$ of such
normalized cycles avoids $Y_2$ nad therefore lifts to 
$(M_1,U_1)$. Together, this implies that $\phi$ induces an isomorphism
$\Omega_j^{(Z,E)}(X_1,Y_1)\to \Omega_j^{(Z,E)}(X_2,Y_2)$.

 As they stand, the bordism groups $\Omega^{(Z,E)}_j(X,Y)$ are rather far from equivariant $K$-homology groups, most obviously because they are not $2$-periodic.   We shall obtain the technical groups $k^G_j(X,Y)$ by simultaneously forcing periodicity and removing dependence of the bordism groups on the pair $(Z,E)$.

Let  $M$ be a stable $(Z,E)$-manifold with structure maps $h$ and $\varphi$,
as in Definition~\ref{ZEmfld-def} and let $F$ be a complex hermitian
$G$-bundle on $Z$ of rank $k$.  The pullback of $F$ to $M$   has  a unique
$G$-invariant smooth structure by Lemma \ref{bijections-lemma}, and so we may form the sphere bundle $S(\R\oplus h^*F)$, which is a $G$-compact proper $G$-manifold.  It is also a stable $(Z,E\oplus F)$-manifold. Indeed if $B(\R\oplus h^*F)$ is the unit ball bundle, and if $p$ is the projection  to $M$, then 
$$  TB(\R\oplus h^*F )  \cong  \R\oplus p^*T^*M\oplus h^*F  $$
(once a complement to the vertical tangent bundle is chosen).  So we obtain an isomorphism
$$\begin{aligned}
 \R^r\oplus TB(\R\oplus h^*F )   &\cong  \R^r \oplus  \R\oplus p^*T^*M\oplus p^* h^*F  \\
  &\cong  \R^s \oplus  \R\oplus p^*h^* E \oplus p^*h^* F  \end{aligned}
$$
using the given stable $(Z,E)$-structure on $M$.  We can then equip the sphere bundle with the stable $(Z,E\oplus F)$-structure it inherits as the boundary of the ball bundle.

Suppose now that $(M,a,f)$ is a cycle for the bordism group $ \Omega^{(Z,E)}_j(X,Y) $.  We can form from it the cycle $(\widehat M, \iota_! (a), f\circ \pi )$ for the group 
$ \Omega^{(Z,E\oplus F)}_{j+2k}(X,Y) $, where:
\begin{enumerate}[\rm (a)]
\item $\widehat M$ is the sphere bundle for $\R\oplus h^* F$ with the stable $(Z,E\oplus F)$-structure just described.
\item  $\iota \colon M\to \widehat M$ is the inclusion of $M$ into the sphere
  bundle given by the formula $m\mapsto (1,0)\in \R\oplus F_{h(m)}$ and
  $\iota_!\colon K_G^0(M)\to K_G^0(\widehat M)$ is the Gysin
  map, as described at the end of Section~\ref{cstar-sec}.
\item $\pi $ is  the projection from  $\widehat M$  to $M$.
\end{enumerate}
Since this construction may also be carried out on bordisms between  manifolds, we obtain a well-defined map on bordism classes. 

\begin{definition}
Let $k=\dim_{\C} (F)$.  Denote by 
$$
\beta^F\colon \Omega^{(Z,E)}_j(X,Y) \longrightarrow \Omega^{(Z,E\oplus F)}_{j+2k}(X,Y)$$
the map determined by the above construction.
\end{definition}

\begin{lemma}
If $F_1$ and $F_2$ are $G$-bundles on $Z$ of ranks $k_1$ and $k_2$ respectively, then 
$$
\beta^{F_2}\circ \beta^{F_1}=\beta^{F}\colon 
 \Omega^{(Z,E)}_j(X,Y) \longrightarrow \Omega^{(Z,E\oplus F)}_{j+2k}(X,Y),$$
 where $F=F_1\oplus F_2$ and $k=k_1+k_2$.
\end{lemma}

\begin{proof}
 Let $c=(M,a,f)$ be a cycle for the bordism group $ \Omega^{(Z,E)}_j(X,Y)$.  
 The manifolds {$\widehat M_0$}, $\widehat M_1$ and $\widehat
 M_2$ obtained from $M$ by the modification processes underlying
{$\beta^{F_1}$,} $\beta^{F_2}\circ \beta^{F_1}$ and $\beta^{F}$ are 
 fiber bundles over $M$ whose fibers are the {spheres in $\R\oplus F_1$ in the
 first case,} the product of spheres in $\R\oplus F_1$ and $\R\oplus F_2$ in the {second} and the sphere in $\R\oplus F$ in the {third}. The product embeds as a hypersurface in the unit ball of $\R\oplus F$, for example via the map 
 $$
 \bigl ( (s_1,f_1), (s_2,f_2)\bigr )\mapsto \tfrac 16\bigl ( (2+s_1)s_2,f_1,(2+s_1)f_2\bigr),
 $$
 and we obtain from this construction a bordism $\widehat W$ between $\widehat M_1 $ and $\widehat M_2$.  The map 
 $$ j\colon t\mapsto  ( \tfrac 12 (1+t),0,0\bigr)
 $$
embeds $M{\times} [0,1]$ into $\widehat W$, transversely to the boundary of
$\widehat W$, and on the boundary components of $M{\times} [0,1]$ the
embedding restricts to the given embeddings of $M$ into $\widehat M_1$ and
$\widehat M_2$.  {The embedding into $\widehat M_1$ is the
  composition of the embedding into $\widehat M_0$ with the embedding of
  $\widehat M_0$ into $\widehat M_1$.} The class $a\in K^0_G(M)$ determines a
class  $a\in K^0_G(M{\times}[0,1])$ by homotopy invariance, and the triple $
(\widehat W, j_{!}(a), f\circ \pi)$ gives a bordism between the cycles
representing  $\beta^{F_2}(\beta^{F_1}(c))$ and $\beta^{F}(c)$, as
required, using functoriality of the Gysin homomorphism to
  describe $\beta^{F_2}(\beta^{F_1}(c))$ in terms of the embedding
  $M\hookrightarrow \widehat M_1$.
\end{proof}

Now fix a universal space $\underline E G$ as in Section~\ref{proper-sec}.   
Let $Z$ be a $G$-finite, $G$-sub\-complex   of $\underline {E} G$.  In order to cope with the contingency that $G$ might be finite we shall modify the notion of $G$-standard bundle as advertised in Remark~\ref{modification-remark}, so that standard $G$-bundles are now taken to be suitable subbundles of $Z\times \C[G_\infty]$.  

Let $E$ be a standard $G$-bundle over $Z$, as considered in Section~\ref{vb-sec}.    Form a partial order on the set of pairs $(Z,E)$ by inclusion:
$$
(Z_1,E_1)\le (Z_2,E_2)\quad \Leftrightarrow  \quad
Z_1\subseteq Z_2 \quad \text{and}\quad  E_1 \subseteq E_2\vert_{Z_1}.
$$
 According to the results of Section~\ref{vb-sec} this is a directed set.

\begin{definition}
For $j\in \Z$ define  groups $k^G_j(X,Y)$ to be the direct limits 
$$k^G_j(X,Y) = \varinjlim_{(Z,E)} \Omega^{(Z,E)}_{j+2\rank(E)}(X,Y)$$
over the directed set of all pairs $(Z,E)$, as above.
\end{definition}

\section{Proof of the Main Theorem}
\label{proof-sec}

We aim to prove Theorem~\ref{main-theorem}, that the geometric equivariant $K$-homology groups of Section~\ref{geometric-sec} are isomorphic to the analytic groups of Section~\ref{kk-review-sec}.  We shall do so by comparing the technical groups of the previous section first to    equivariant $KK$-theory and then to geometric $K$-homology.

The equivariant $KK$-groups determine a homology theory on $G$-finite proper $G$-$CW$ pairs (or indeed on arbitrary second-countable $G$-compact proper $G$-$CW$ pairs) if one defines the relative groups for a pair $(X,Y)$ to be $KK_j^G(C_0(X\setminus Y),\C)$.  The boundary maps are provided by the boundary maps of the six-term exact sequence in $KK$-theory. 

If $(M,a, f)$ is a cycle for $\Omega^{(Z,E)}(X,Y)$, then an element of the Kasparov group $KK_j^G(C_0(X\setminus Y),\C)$ may be defined as follows. Form the Dirac operator $D$ on the interior of $M$ using the $\Spin^c$-structure associated to the given stable  $(Z,E)$-structure on $M$.
It determines a class 
$$
[D] \in KK^G_j (C_0(M\setminus \partial M),\C) .
$$
For example we may equip $M\setminus \partial M)$ with a complete $G$-invariant Riemannian metric and then form a $KK$-class using $F= FD(I+D^2)^{-\frac 12}$ as in Section~\ref{kk-review-sec} (it   does not depend on the choice of metric). 
Compare \cite[Ch.~10]{MR2002c:58036}, where the non-equivariant case is
handled; the $G$-compact proper $G$-manifold situation is the same.  We can
then form the Kasparov product $a\otimes [D]\in
KK_j^G{(C_0(M\setminus \partial M),\C)}$
and hence the class  
$$
f_*(a\otimes [D])\in KK_j^G(C_0(X\setminus Y),\C)  
$$
  more or less exactly as we did in Section~\ref{kk-review-sec}.
  
\begin{theorem} The correspondence that associates to a cycle $(M,a,f)$ the element $f_*(a\otimes [D])$ above is a natural transformation 
$$
{\mu\colon} k_*^G(X,Y) \longrightarrow KK_*^G(C_0(X\setminus Y),\C)
$$
between homology theories. \qed
\end{theorem}

 This is a mild elaboration of Theorem~\ref{mu-well-defined} and the equivariant counterpart of  \cite[Theorem 6.1]{MR2330153}.  The equivariant case may be handled exactly as in \cite{MR2330153}.

Our first goal {is} to show that this natural transformation 
is an isomorphism on $G$-finite proper $G$-$CW$-complexes.  To do so it
suffices to show that it is an isomorphism on the {building
  blocks} $X=G/H$ corresponding to finite subgroups  $H$  of $G$.   The
following observation clarifies what needs to be done in this case.
 
 \begin{lemma}
 Let $H$ be a finite subgroup of $G$.
There is a commutative diagram 
$$
\xymatrix{
k_*^G (G/H)\ar[r]^-{\mu} &  KK_*^G(C_0(G/H),\C) \\
k_*^H (\Pt)\ar[r]_{\mu}\ar[u]^{\Ind} &  KK_*^H(\C,\C)\ar[u]_{\Ind}
}
$$
in which  the vertical maps are isomorphisms.  
\end{lemma}

\begin{proof}
The right-hand induction map  is defined as follows.   If $\mathcal H$ is a Hilbert space, or Hilbert module, with unitary $H$-action, define $\Ind \mathcal H$ to be the space of square-integrable sections of  $G\times_H\mathcal H$.  It carries a natural representation of $C_0(G/H)$, and if 
$F$ is an $H$-equivariant Fredholm operator on $\mathcal H$, then operator $\Ind F$ on $\Ind \mathcal H$ given by the pointwise action of $F$ determines a cycle for $KK_j^G (C_0(G/H),\C)$.  

The inverse to the induction map defined in this way is given by compression to the range of  the projection operator determined by the indicator function of the identity coset in $G/H$ (this function being viewed as an element of $C_0(G/H)$).

The left-hand induction map is defined in a similar fashion.  We choose our model for $\underline E H$ to be a point, which we can include into $\underline E G$ as an $H$-fixed $0$-cell, and we use  the induced manifolds  $\Ind M = G\times_H M$, which map canonically to $G/H\subseteq \underline E G$.  Note that any $G$-manifold that maps to $G/H$ has this form. The construction of an inverse to induction  and the proof that induction is an isomorphism  are immediate upon noting that any $G$-map $f\colon \Ind M\to  \underline  E G$ is $G$-equivariantly homotopic to one that factors through $G/H\subseteq \underline E G$. \end{proof}

It suffices, therefore, to prove that the   map 
$$
\mu \colon k^H_j(\Pt)\longrightarrow KK^H_j (\C,\C)
$$
is an isomorphism.  The right hand group is isomorphic to the representation ring $R(H)$ when $j$ is even and is zero when $j$ is odd.

\begin{lemma}
\label{fried-egg-lemma}
Let $(M,a,f)$ be a cycle for $k_j^{G} (X)$.  If $a=a_1+a_2$ in $K^0_G(M)$, then 
$$
[(M,a,f)] = [(M,a_1,f)] + [(M,a_2,f)]
$$
in $k_j^{G} (X)$.
\end{lemma}

\begin{proof}
Suppose that $(M,a,f)\in  \Omega_{j+2k}^{(Z,E)}(X)$.  Fix a bordism $W$ between $S^2$ and $S^2\sqcup S^2$ by situating two copies of the $2$-sphere of radius $\frac 14 $  inside the unit sphere. There are smooth paths $I_1$ and $I_2$ embedded into the bounding manifold $W$ that connect the north and south poles of the large sphere to the south and north poles of the small spheres, that meet the spheres transversally, and that have trivial normal bundles in $W$.  
\begin{center} 
 \includegraphics[height=3cm]{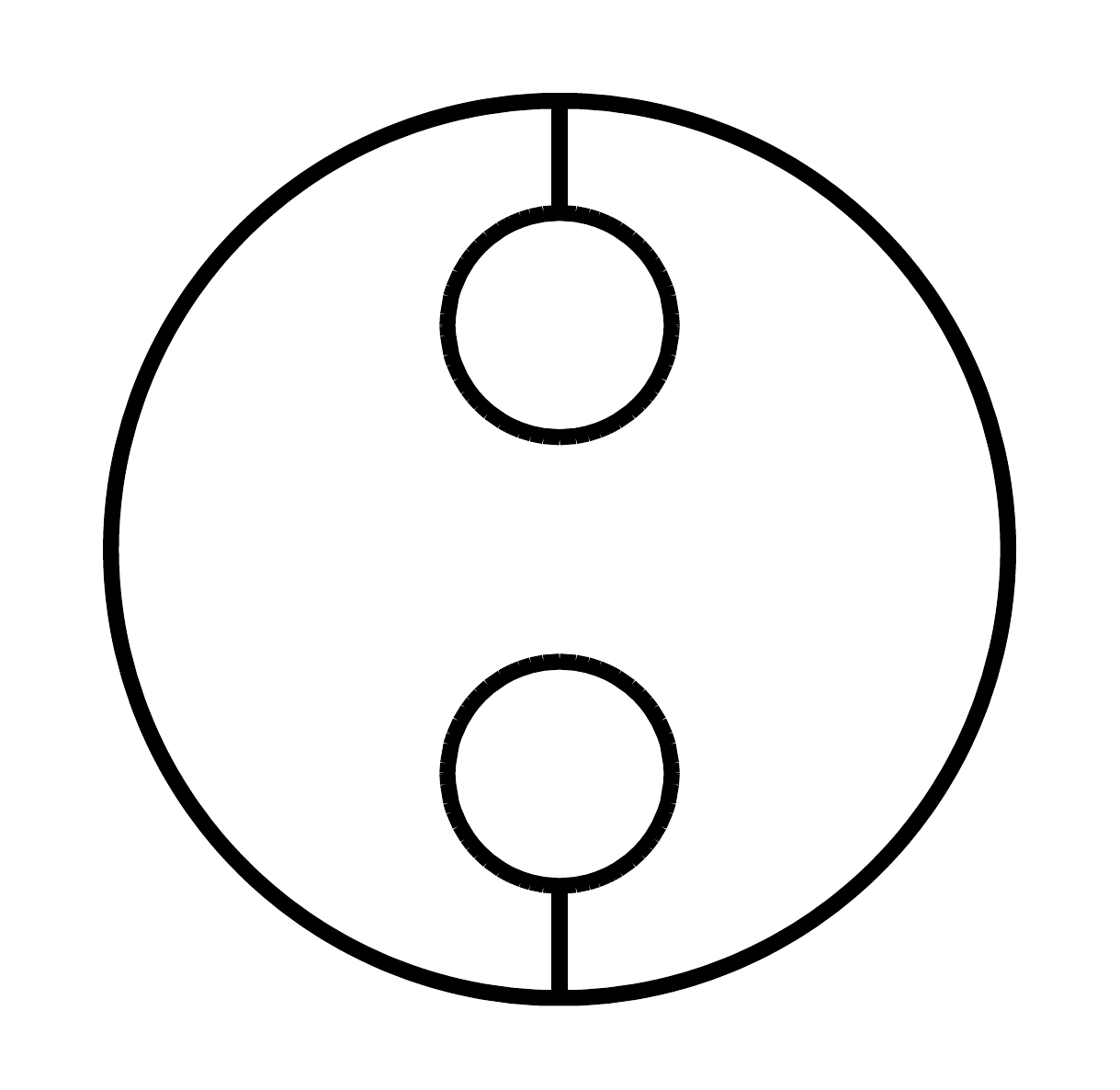}  
 \end{center}

\noindent
The class $a_1\in K_G ^0(M)$ pulls back to a class $\widetilde a_1 \in K_G ^0(M\times I_1)$.  Similarly, the class $a_2$ pulls back to a class $\widetilde a_2 \in K_G^0(M\times I_2)$.  We obtain 
$$\widetilde a := \widetilde a_1 \sqcup \widetilde a_2 \in K_G ^0(M\times I)=
K_G^0(M\times I_1)\oplus K_G^0(M\times I_2),
$$
where $I=I_1\sqcup I_2$.
Now  form the class 
$$j_!  ( \widetilde a ) \in K_G^0(M\!\times \!W),$$
 where $j$ is the inclusion of $M\times I$ into $M\!\times\! W$.
If $\tilde f\colon M\times W \to X$ is the projection from $M\times W$ to $M$,
followed by $f$, then $(M{\times} W, j_!(\tilde a), \tilde f )$ is a   bordism
between the images of $(M,a,f)$ and $(M,a_1,f) \sqcup  (M,a_2,f)$ under the
map $$\beta\colon \Omega_j(X)^{(Z,E)}\to \Omega_{j+2(k+1)}^{(Z,E\oplus
  \C)}(X).$$ 
 For this, observe that 
$$\beta(M,a,f)=(M\times
  S^2,i^N_!(a_1+a_2),f)= (M\times S^2,i^N_!(a_1)+i^S_!(a_2),f)$$
where $i_N,i_S\colon M\to M\times S^2$ are the north pole inclusion and the
south pole inclusion, which are $G$-homotopic, and the Gysin map is homotopy
invariant. Finally, $i_!^N(a_1)=j_!(\tilde{a}_1)|_{M\times S^2}$ by the
compatibility of the Gysin map with restriction, so that
$i^N_!(a_1)+i^S_!(a_2)=j_!(\tilde a)|_{M\times S^2}$. 

In view of the definition of $k^G_j(X)$ the lemma is proved.
\end{proof}

\begin{remark} The lemma shows that $-[(M,a,f)] = [(M,-a,f)]$ in $k^G_j(X)$.
\end{remark}

\begin{proposition} 
Let $H$ be a finite group.
The homomorphism   
$$\mu \colon k^H_*(\Pt) \longrightarrow KK^H_*(\C,\C)$$
 is an isomorphism.  
\end{proposition}

\begin{proof}
We shall prove that the homomorphism   
\begin{gather*}
R(H)\longrightarrow k^H_0(\Pt)\\
a\mapsto (\Pt, a, \Id)
\end{gather*}
 is surjective, and that in  addition, the group $k^H_1(\Pt)$ is zero.  This will suffice since $KK_0^H(\C,\C)\cong R(H)$, with the isomorphism being given by the above correspondence and the map $\mu$, while $KK_1^H(\C,\C)=0$.

Fix an element of    $k^H_0(\Pt)$ and represent it by a  cycle $(M,a)$ for
$\Omega_{2n} ^{(\Pt,E)}(\Pt)$ (we drop the map $f\colon M\to \Pt$ from our
notation here and below).  Thus $M$ is a $2n$-dimensional $\Spin^c$-manifold
with a given stable isomorphism from its tangent bundle to a trivial bundle
$M{\times} E$, where $E$ is a complex   representation of $H$.  The manifold
$M$ may be equivariantly embedded in  a finite-dimensional complex
representation $V$  of $H$. By composing with the subspace embedding $V\to
E{\oplus} V$  and by adding to $V$ a multiple of the {regular}
representation, if necessary, we arrive at an embedding into $E{\oplus }V$
with trivial normal bundle $F:= M{\times}V$.

Using the map
$$
\beta^F\colon \Omega^{(\Pt,E)}_{2n}(\Pt) \longrightarrow \Omega^{(\Pt,E\oplus V)}_{2n+2k}(\Pt)
$$
we find that the given  element of  $k^H_0(\Pt)$ is represented by the cycle
$(N,b) =(\widehat M, \iota_*(a))$.  Here   $N=\widehat M$ is a codimension one submanifold of  $E{\oplus}V{\oplus} \R$   and is the boundary of  some compact $X$ (namely the ball bundle associated to $\widehat M$).    By enclosing $X$ in a large ball we construct a bordism $Y$ between $\widehat M$ and a sphere $S\subseteq V{\oplus} F{\oplus} \R$.  The union of $ X $ and $  Y$ is the ball bounded by the sphere $S$,  and by applying the Mayer-Vietoris sequence in $H$-equivariant $K$-theory to the decomposition $ X \cup   Y$ of the ball, we find that the class $b\in K^0_H(N)$ may be written as a sum $b_X+b_Y$, where $b_X$ is a the restriction of a class in $K^0_H(X)$ and $b_Y$ is a the restriction of a class in $K^0_H(Y)$.  By the previous lemma, 
$$
[(N,b)] = [(N,b_X)] + [(N,b_Y)].
$$
The first class is zero in $ \Omega^{(\Pt,E\oplus F)}_{2n+2k}(\Pt)$, while the second is equal to some class $[(S,c)]$ thanks to the bordism $Y$.  We have therefore shown that every class in $k_0^H(\Pt)$ represented by a sphere in some $W{\oplus} \R$, where $W$ is a complex representation of $H$.

To complete the proof we invoke Bott periodicity.  The class $c$ is a sum $c_0 + c_1$ where $c_0$ is represented by a trivial bundle (one that extends over the ball) and $c_1$ is in the image of the Gysin map associated to the inclusion 
$$\Pt\mapsto (0,1)\in   S\subseteq W\times \R.$$
This completes the computation of $k^H_0(\Pt)$.  

The proof that $k_1^H(\Pt) = 0$ is essentially the same.  Every cycle is equivalent to one of the form $[(S,c)]$, where $S$ is the sphere in a complex representaion of $H$.  But the sphere is odd-dimensional and by Bott periodicity every class in $K_H^0(S)$ extends over the ball.  So $(S,c)$ is a boundary.
\end{proof}

\begin{corollary}
If $X$ is a $G$-finite proper $G$-$CW$-complex, then the map 
$$
\mu\colon k_*^G(X)\longrightarrow KK_*^G(C_0(X),\C)
$$
is an isomorphism.
\qed
\end{corollary}
  Next we need to relate the technical groups $k_j^G(X)$ to the geometric groups $K_j^G(X)$.

\begin{lemma}
The class in $K^G_j(X)$ of an equivariant $K$-cycle $(M,E,f)$ depends only on $M$, $f$ and the class of the $G$-bundle $E$ in the Grothendieck group $K^0_G(X)$.
\end{lemma}

\begin{proof}
Fixing $M$ and $f$, the map that associates to a $G$-bundle $E$ on $M$ the class of $(M,E,f)$ in $K^G_j(X)$ is additive, and so extends to a map from the Grothendieck group $K^0_G(X)$ into $K^G_j(X)$.
\end{proof}

Thanks to the lemma, we can attach a meaning to the class in the geometric group $K_j^G(X)$ of any triple $(M,a,f)$, whenever   $M$ is a $G$-compact proper $G$-$\Spin^c$-manifold, $f$ is an equivariant map from $M$ to $X$, and $a\in K_G(M)$.  In particular, we can do so when $M$ is a stable $(Z,E)$-manifold.  We obtain in this way a natural transformation 
$$
\Omega^{(Z,E)}_j(X)\longrightarrow K^G_j(X)
$$
\begin{lemma}
If $F$ is any $G$-bundle over $Z$ of rank $k$, then the diagram 
$$
\xymatrix@C=50pt{
 \Omega^{(Z,E)}_j(X)\ar[r] \ar[d]_{\beta^F} &   K^{G}_j (X)\ar@{=}[d]\\
\Omega^{(Z,E\oplus F)}_{j+2k}(X) \ar[r] &K^G_{j+2k}(X)}
$$
is commutative.
\end{lemma}

\begin{proof}
This follows from the definitions of the Gysin homomorphism and vector bundle modification.
\end{proof}

We obtain therefore a natural transformation
$$
k^G_j(X)\longrightarrow K^G_j(X)
$$
that fits into a commuting diagram 
$$
\xymatrix@C=50pt{
  k^{G}_j (X)\ar[r]\ar@{=}[d] &   K^{G}_j (X)\ar[d]^-{\mu}\\
 k^{G}_j(X) \ar[r]_-\mu&KK^{G}_j (C_0(X ),\C).}$$

\begin{lemma}
For every $j$ the map from $k^G_j(X)$ into the equivariant geometric $K$-homology group $K_j^G(X)$ is surjective.
\end{lemma}

\begin{proof}
Let $(M,a,f)$ be an equivariant $K$-cycle for $X$.  The manifold $M$ maps
equivariantly to a $G$-finite subcomplex $Z$ of the universal space
$\underline E G$  {via $h$}.

 The tangent bundle for $M$ (indeed its complexification) embeds as a summand
of a    $G$-bundle that is pulled back from a standard $G$-bundle on  $Z$ via  $h\colon M\to Z$  (see Corollary~\ref{dir-summand-lemma} and the comment following it). Thus there is an isomorphism of real bundles
$$
TM\oplus F_0 \cong  h^* E
$$
where {$F_0$} is a real $G$-bundle on $M$ and $E$ is a complex $G$-bundle on $Z$.  By adding a trivial bundle if necessary, we obtain an isomorphism
$$
TM\oplus F_1 \cong  h^* E \oplus \R^s,
$$
where $F_1=F_0\oplus \R^s$   has \emph{even} fiber dimension.  By the two out of three principle for $\Spin^c$-structures from Section~\ref{geometric-sec},  the bundle $F_1$ carries a $\Spin^c$-structure whose direct sum with the given $\Spin^c$-structure on $TM$ is the direct sum of the $\Spin^c$-structure on $h^*E$ associated with its complex structure and the trivial $\Spin^c$-structure on $\R$.  If we carry out a  vector bundle modification using $F_1$, then we obtain an  equivariant $K$-cycle $(M, E, f)\hat{\phantom{tt}}$ that is equivalent to $(M,E,f)$ and for which $\widehat M$ carries a stable $(Z,E)$-structure compatible with its $\Spin^c$-structure.
\end{proof}

With this, as we pointed out in the introduction, the proof of Theorem~\ref{main-theorem} is complete.

\section{The Baum-Connes Assembly Map}

Let $G$ be a countable discrete group.  The essence of the Baum-Connes
conjecture for $G$ is the assertion  that every class in the $K$-theory of the
reduced group $C^*$-algebra $C^*_r(G)$ arises as the index of an elliptic
operator on a $G$-compact proper $G$-manifold, and that in addition the only
relations among these indexes arise from geometric relations (such as for
example bordism) between the operators.   The conjecture arose from a
$K$-theoretic  analysis of Lie groups and of  crossed product algebras related
to foliations.   However here we shall discuss only the   $C^*$-algebras of
discrete groups.

To make their conjecture precise, Baum and Connes constructed in \cite{MR1769535} geometric groups $K^j(G)$ from cycles related to the symbols of equivariant elliptic pseudodifferential operators, and an equivalence relation related to the Gysin map in $K$-theory.  They then defined an index map 
$$\mu \colon K^j(G)\longrightarrow K_j(C^*_r (G))
$$
that they conjectured to be an isomorphism.

Although the Dirac operator on a $\Spin^c$-manifold did not play a central role in \cite{MR1769535}, it is nonetheless a fairly routine matter to identify the geometric group  defined there with the group generated from cycles $(M,E)$, where:
\begin{enumerate}[\rm (a)]
\item $M$ is a $G$-compact proper $G$-$\Spin^c$-manifold, all of whose components have either even or odd dimension, according as $j$ is $0$ or $1$, and 
\item $E$ is a complex $G$-bundle on $M$.
\end{enumerate}
The equivalence relation between cycles is generated by bordism, direct sum/dis\-joint union and vector bundle modification, exactly as in Section~\ref{geometric-sec}, except that here there is no reference space $X$, nor any map from $M$ to $X$.  Compare \cite{MR2106816}, where the conjecture for countable discrete groups is formulated in precisely this way.

It follows from the universal property of the space $\underline E G$ that there is an isomorphism
$$K^j(G) \cong  \varinjlim _{X\subseteq \underline E G} K_j^G(X),$$
where on the right is the direct limit of the geometric $K$-homology groups of the $G$-finite subcomplexes of the $G$-$CW$-complex $\underline E G$.  

In a later paper \cite{MR1292018}, Baum, Connes and Higson defined an assembly map 
$$
\mu\colon \varinjlim _{X\subseteq \underline E G} KK_j^G(C(X),\C) \longrightarrow K_j(C^*_r (G))
$$
Its relation to the original Baum-Connes map is summarized by the commutative
diagram 
$$
\xymatrix{
K^j(G)\ar[r]^-{\cong}\ar[d]_{\mu}&c\varinjlim _{X\subseteq \underline E G}  K_j^G(X)  \ar[r]^-{\mu}& \varinjlim _{X\subseteq \underline E G}KK_j^G(C(X),\C)\ar[d]^{\mu} \\
  K_j(C^*_r (G))\ar[rr]_{=}&&   K_j(C^*_r (G)) ,
}
$$
where the horizontal map labelled (yet again) $\mu$ is the one analyzed in this paper, and shown to be an isomorphism.   Because it is an isomorphism, the reformulation of the Baum-Connes conjecture in \cite{MR1292018} is equivalent to the original in \cite{MR1769535} for discrete groups.

Despite the discovery some years ago of   counterexamples to various extensions of the Baum-Connes conjecture (see \cite{MR1911663}), there is, as of today, no known counterexample to the conjecture as reviewed here.

\bibliography{BaumHigsonSchick_References}	

\noindent P.B.:  Department of Mathematics,  Penn State University, University Park, PA 16802. Email: {\tt baum@math.psu.edu}.

\noindent N.H.:  Department of Mathematics,  Penn State University, University Park, PA 16802. Email: {\tt higson@math.psu.edu}.

\noindent T.S.:  Mathematisches Institut, Universit\"at G\"ottingen,
Bunsenstr.\ 3, 
D-37073 G\"ottingen, Germany. Email:  {\tt thomas.schick@uni-math.gwdg.de}.

\end{document}